\documentclass[12pt]{article}
\usepackage{enumerate}
\usepackage{amsmath}
\usepackage{amsthm}
\usepackage{amssymb}
\usepackage{algorithm}
\usepackage[bookmarksnumbered,  plainpages]{hyperref}

\newtheorem{thm}{Theorem}[section]

\newtheorem{lem}{Lemma}[section]
\newtheorem{mthd}{Method}[section]

\theoremstyle{definition}
\newtheorem{defn}{Definition}[section]
\theoremstyle{remark}
\newtheorem{rem}{Remark}[section]
\begin{document}
	
	\title{{More on Modulus Based Iterative Method for Solving Implicit Complementarity  Problem}}
	
	\author{Bharat kumar$^{a,1}$, Deepmala$^{a,2}$ and A.K. Das $^{b,3}$\\
		\emph{\small $^{a}$Mathematics Discipline,}\\
		\emph{\small PDPM-Indian Institute of Information Technology, Design and Manufacturing,}\\
		\emph{\small Jabalpur - 482005 (MP), India}\\
		\emph{\small $^{b}$Indian Statistical Institute, 203 B.T. Road, }\\
		\emph{\small Kolkata - 700108, India}\\
		%\emph{\small $^{b}$}\\
		\emph{\small $^1$Email:bharatnishad.kanpu@gmail.com , $^2$Email: dmrai23@gmail.com}\\
		\emph{\small $^3$Email: akdas@isical.ac.in}}
	\date{}
	\maketitle
	
	\abstract{\noindent This article presents a class of modified new modulus-based iterative methods to process the large and sparse implicit complementarity problem (ICP).  By using two positive diagonal matrices, we formulate a fixed-point equation which is equivalent to an ICP  and  based on a fixed-point equation, an   iterative method is presented to solve the ICP.  We provide some convergence conditions for the proposed methods when the system matrix is a $P$-matrix or an $H_+$-matrix.}

	
	
	
	\maketitle
	
	\section{Introduction}\label{sec1}
	The implicit complementarity problem (ICP) was introduced by Bensoussan et al. in 1973. For detail, see \cite{7}, \cite{8}.  The linear complementarity problem (LCP) is a special case of the ICP.  The ICP is used in many fields, such as engineering and economics, scientific computing, stochastic optimal control problems and convex cones. The ICP has been discussed and studied in some literature. For details, see \cite{1}, \cite{2}, \cite{3}, \cite{4}, \cite{5}, \cite{6}.\\
	Consider  matrix $A \in \mathbb{R}^{n\times n} $ and  vector $q \in \mathbb{R}^n$. The implicit complementarity problem denoted as ICP$(q, A, \psi)$ is to find the solution $ z \in \mathbb{R}^n $ to the following system:
	\begin{equation}\label{eq1}
		p(z)=Az+q\geq 0, ~~~r(z)=z-\psi(z)\geq 0, ~~~ p(z)^Tr(z)=0,
	\end{equation}
	where $\psi(z)$ is a mapping from $\mathbb{R}^n$ into $\mathbb{R}^n$ such that $r(z)$ is invertible. When we put $\psi(z)=0$ in Equation (\ref{eq1}), the ICP$(q, A, \psi)$ reduces to the linear complementarity problem (LCP). For more details on LCP and its applications, see \cite{9}, \cite{10}, \cite{35}, \cite{36}, \cite{37}, \cite{38}, \cite{39},  \cite{41}, \cite{42}, \cite{43}, \cite{44}, \cite{45}, \cite{46}.\\
	In recent decades to solve complementarity problems, many methods have been proposed by considering some equivalent problems \cite{26}, \cite{27}, \cite{28}, such as fixed point approaches \cite{11}, projection-type methods \cite{12}, smooth and nonsmooth Newton methods \cite{13}, \cite{14}, linearization methods \cite{15}, domain decomposition methods \cite{16}, matrix multisplitting methods \cite{17}, \cite{18}, \cite{19}, \cite{20}, \cite{21}, \cite{34} and inexact alternating direction methods \cite{22}, \cite{23}, \cite{24}, \cite{25}.
	In 2016, Hong and Li \cite{29}  demonstrated numerically the superiority of the modulus-based matrix splitting (MMS) method over the projection fixed-point method and the Newton method when solving the ICP$(q, A, \psi)$.
	Shilang and Liang \cite{30} presented a family of new modulus-based matrix splitting (NMMS) iteration methods to solve the ICP$(q, A, \psi)$ in 2022. The NMMS method is not only different from the inner-outer iteration methods. In addition, the NMMS method is simple and efficient in its implementation.
	Motivated by the work of Shilang and Liang, we propose a class of modified new    modulus-based iterative methods based on new matrix splitting for solving ICP$(q, A, \psi)$ which covers the NMMS method.
	
	This article is presented as follows: Some useful definitions, notations and well-known lemmas are given in Section 2 which are required for the remaining sections of this work. In Section 3, a family of  modified new modulus-based matrix splitting methods is constructed using the new equivalent fixed-point form of the ICP$(q, A, \psi)$. In Section 4, we establish some convergence criteria for the proposed methods. Section 5 contains the conclusion of the article.
	%%%%%%%%%%%%%%%%%%%%%%%%%%%%%%%%%%%%%%%%%%%%%%%%%%%%%%%%%%%%%%%%%%%%%%%%%%%%%%%%
	\section{Preliminaries}\label{Preli}
	In this section, we provide some basic notations, definitions and lemmas that will be used throughout the  article.\\
	Let  $A=({a}_{ij})\in \mathbb{R}^{n\times n}$, $B=({b}_{ij})\in \mathbb{R}^{n\times n}$.  We use $A \geq (> ) B$ to   denote $a_{ij}\geq (>)  b_{ij}$ $\forall$ $ i,j$ and    $A^{T}$  denotes the transpose of the matrix  $A$. Let  $A=0 \in \mathbb{R}^{n \times n}$  defined as  $ a_{ij}=0 ~\forall~ i,j$  and  $\lvert A\rvert=(c_{ij})$   defined by $ c_{ij}  = \lvert {a}_{ij}\rvert$ $\forall ~i,j$.
	
	\begin{defn}\cite{31}
		Suppose  $ A=(a_{ij})\in  \mathbb{R}^{n\times n}$. Then its  comparison matrix $\langle A \rangle =(\langle a_{ij} \rangle)$  is defined by $\langle a_{ij} \rangle$ = $\lvert {a}_{ij}\rvert$ \text{if $i=j$} and $-\vert {a}_{ij}\rvert$ \text{if $i\neq j$} for $i,j = 1,2,\ldots,n$.
	\end{defn}
	%%%%%%%%%%%%%%%%%%%%%%%%%%%%%%%%%%%%%%%%%%%%
	\begin{defn}\label{def1}\cite{32}
		Suppose  $ A \in \mathbb{R}^{n\times n}$. Then $A$ is said to be 
		\begin{enumerate}
			\item a $Z$-matrix if all of its off-diagonal elements are nonpositive.
			\item  an $M$-matrix if $A^{-1}\geq 0$ as well as   $Z$-matrix.
			\item     an $H$-matrix, if   $\langle A \rangle$ is an $M$-matrix.
			\item   an $H_+$-matrix if it is an $H$-matrix as well as  $ {a}_{ii} > 0$ for $~i =1,2,\ldots,n$.
		\end{enumerate}
	\end{defn} 
	%%%%%%%%%%%%%%%%%%%%%%%%%%%%%%%%%%%%%%%%%
	%%%%%%%%%%%%%%%%%%%%%%%%%%%%%%%%%%%%%%%%%%%%%%%%%
	
	%%%%%%%%%%%%%%%%%%%%%%%%%%%%%%%%%%%%%%%%%%%%%
	
	%%%%%%%%%%%%%%%%%%%%%%%%%%%%%%%%%%%%%%%%%%%
	
	%%%%%%%%%%%%%%%%%%%%%%%%%%%%%%%%%%%%
	
	\begin{defn}{\cite{30}}
		Let $A=(a_{ij})\in \mathbb{R}^{n \times n}$. Then $A $ is said to be a strictly diagonally dominant (sdd) matrix  if $\lvert a_{ii}\rvert>\sum_{j\neq i}\lvert a_{ij}\rvert$ for $i,j=1,2,\ldots,n.$
	\end{defn}
	%%%%%%%%%%%%%%%%%%%%%%%%%%%%%%%%%%%%
	\begin{lem}\label{lem5} \cite{33}
		Let $x, y \in \mathbb{R}^n.$ Then $x\geq 0$, $y\geq 0$ and $x^Ty=0$ if and only if $x+y=\lvert x-y\rvert.$
	\end{lem}
	\begin{lem} \cite{30}\label{2.4}
		Let $A\in \mathbb{R}^{n \times n}$ be a sdd matrix. Then $$\|A^{-1}E\|_\infty \leq \max_{1\leq i \leq n}\frac{(\lvert E \rvert e)_{i}}{(\langle A \rangle e)_{i}},~~ \mbox{for all}~ E \in \mathbb{R}^{n \times n},$$ where $e=(1,1,\ldots,1)^T$.
	\end{lem}
	\begin{lem}\cite{30}
		Let $A \in \mathbb{R}^{n \times n}$ be  a nonsingular $M$-matrix, then there exists a positive diagonal matrix $U$ such that $AU$ is an sdd matrix.
	\end{lem}
	%%%%%%%%%%%%%%%%%%%%%%%%%%%%%%%%%%%%%%%%%%%%%%%%%%%%%%%%%%%%%%%%%%%%
	\section{Main results}\label{sec2}
	In this section, we present a class of modified new modulus-based splitting methods for solving ICP$(q, A,\psi)$.   First, we construct a new equivalent expression of the ICP$(q, A,\psi)$ using the Lemma $\ref{lem5}$. The details are as follows.
	\begin{thm}\label{thm0}
		Let $A \in \mathbb{R}^{n \times n}$ and $q \in \mathbb{R}^{n}$. Suppose $\phi_{1}$, $\phi_{2} \in \mathbb{R}^{n \times n}$  are two positive diagonal matrices. Then the ICP$(q, A,\psi)$ is equivalent to 
		\begin{equation} \label{eq2}
			(\phi_1 A +\phi_{2})z=\lvert{(\phi_1 A-\phi_{2})z+\phi_{1} q+\phi_{2}\psi(z)}\rvert-\phi_{1} q+\phi_{2}\psi(z).
		\end{equation}
	\end{thm}
	\begin{proof}
		The ICP$(q, A,\psi)$ can be rewritten as 
		\begin{equation*}
			\phi_{1}(Az+q)\geq 0, ~~~~~~\phi_{2}(z-\psi(z))\geq 0,~~~~~~~ (\phi_{2}(z-\psi(z))^T(\phi_{1}(Az+q))=0.
		\end{equation*}
		Let $x=\phi_{1}(Az+q)$ and $y=\phi_{2}(z-\psi(z))$. Then by using the Lemma \ref{lem5}, we write 
		\begin{equation*}
			\phi_{1}(Az+q)+\phi_{2}(z-\psi(z)) =\lvert \phi_{1}(Az+q)-\phi_{2}(z-\psi(z))\rvert.
		\end{equation*}
		This implies that
		\begin{equation*}
			(\phi_1 A +\phi_{2})z=\lvert{(\phi_1 A-\phi_{2})z+\phi_{1} q+\phi_{2}\psi(z)}\rvert-\phi_{1} q+\phi_{2}\psi(z).
		\end{equation*}
	\end{proof}
	
	\noindent According to the statement in the Theorem \ref{thm0}, we are able to find the solution to ICP$(q, A,\psi)$ with the use of an implicit fixed-point Equation (\ref{eq2}). Now we provide an iterative method for solving the ICP$(q, A,\psi)$ by using a new matrix splitting of the system matrix $A$.  Suppose $A =(M_{1}+\phi)-(N_{1}+\phi)$ as a matrix
	splitting of $A \in \mathbb{R}^{n\times n} $, where $\phi$ is a diagonal matrix of order $n$. Now from Equation (\ref{eq2}), we obtain
	\begin{equation}\label{eq03}
		\begin{split}
			(\phi_{1} M_1 + \phi_{1} \phi + \phi_{2})z&=(\phi_{1} N_1+\phi_{1} \phi)z+ \lvert{(\phi_{1} A-\phi_{2})z+\phi_{1}q+\phi_{2}\psi(z)}\rvert\\&-\phi_{1}q+\phi_{2}\psi(z).
		\end{split}
	\end{equation}
	Let $\phi_{1}M_1=M_{\phi_{1}}, ~\phi_{1}N_1=N_{\phi_{1}},  ~\phi_{1}A=A_{\phi_{1}}$ and $\phi_{1}\phi=\phi_{\phi_{1}}$. Then the Equation (\ref{eq03}) can be rewritten as 
	\begin{equation}\label{eq4}
		\begin{split}
			(M_{\phi_{1}}+\phi_{\phi_{1}}+\phi_{2})z&=(N_{\phi_{1}}+\phi_{\phi_{1}})z+\lvert(A_{\phi_{1}}-\phi_{2})z+\phi_{1}q+\phi_{2}\psi(z)\rvert\\&-\phi_{1}q+\phi_{2}\psi(z).
		\end{split}
	\end{equation}
	In order to solve the ICP$(q, A,\psi)$ from  Equation ($\ref{eq4}$) we establish a modified new modulus-based matrix splitting iterative method which is known as Method 3.1.\\	
	\begin{mthd}\label{mthd1}
		
		Let $A=(M_1+\phi)-(N_1+\phi)$ be a splitting of the matrix $A\in \mathbb{R}^{n\times n}$ and $q\in \mathbb{R}^{n}$. Suppose that  $M_{\phi_{1}}+\phi_{\phi_{1}}+\phi_{2}$ is a nonsingular matrix where $\phi$ is a diagonal matrix and   $\phi_{1}$, $\phi_{2}$   are two positive diagonal matrices. Then we use the following equation for  Method  \ref{mthd1} is 
		\begin{equation}\label{eq5}
			\begin{split}
				z^{(k+1)}&=(M_{\phi_{1}}+\phi_{\phi_{1}}+\phi_{2})^{-1}[(N_{\phi_{1}}+\phi_{\phi_{1}})z^{(k)}+\lvert(A_{\phi_{1}}-\phi_{2})z^{(k)}+\phi_{1}q\\&+\phi_2\psi(z^k)\rvert-\phi_{1}q+\phi_2\psi(z^k)].
			\end{split}
		\end{equation}
		Let residual be the Euclidean norm of the error vector which is defined in \cite{30} as follows:  $$ Res(z^{(k)})=\|min(z^{(k)}-\psi(z^k), Az^{(k)}+q) \|_{2}.$$ 
		Consider a  nonnegative initial vector $z^{(0)}\in \mathbb{R}^n$. For $k=0,1,2,\ldots,$ the iterative process continues until the iterative sequence $\{z^{(k)}\}_{k=0}^{+\infty} \subset \mathbb{R}^n$ converges. The iterative process stops if $Res(z^{(k)})$ $< $ $ \epsilon $. For  computing  $z^{(k+1)}\in \mathbb{R}^{n}$ we use the following steps.\\
		\noindent \textbf{Step 1}: Initialize a vector $z^{(0)} \in \mathbb{R}^{n}$,  $\epsilon >  0 $  and set  $ k=0 $.\\
		\textbf{Step 2}: Using the following scheme, generate  the sequence $z^{(k)}$:
		\begin{equation*}
			\begin{split}
				z^{(k+1)}&=(M_{\phi_{1}}+\phi_{\phi_{1}}+\phi_{2})^{-1}[(N_{\phi_{1}}+\phi_{\phi_{1}})z^{(k)}+\lvert(A_{\phi_{1}}-\phi_{2})z^{(k)}+\phi_{1}q\\&+\phi_2\psi(z^k)\rvert-\phi_{1}q+\phi_2\psi(z^k)].
			\end{split}
		\end{equation*}
		\textbf{Step 3}: If $ Res(z^{(k)})$ $ < $ $\epsilon$  then stop otherwise set $k=k+1$ and return to step 2. 
		
	\end{mthd}
	\begin{rem}
		If we put  $\phi_{1}=I$, where $I$ is an identity matrix of order $n$ and $\phi=0 \in \mathbb{R}^{n \times n}$ in Equation $(\ref{eq2})$, then the result of Theorem $\ref{thm0}$ reduces to the result \cite{30}.
	\end{rem}
	Furthermore, Method \ref{mthd1}  provides a general framework for solving ICP$(q, A, \psi)$.  We obtain  a class of modified new modulus-based matrix splitting iterative methods. In particular, we express the system matrix $A$ as $A=(M_1+\phi)-(N_1+\phi)$. Then 
	\begin{enumerate}
		\item when $M_{1}=A,~ \phi_1 = I, ~\phi_2 =I,~ N_{1}=0 $, from Equation ($\ref{eq5}$) we obtain 
		\begin{equation*}
			(A+\phi+I)z^{(k+1)}=\phi z^{(k)}+\lvert(A-I)z^{(k)}+q+\psi(z^{(k)})\rvert-q+\psi(z^{(k)}),	
		\end{equation*}
		this is known as the ``modified  new modulus-based iterative (MNMOD) method$"$.
		\item when $M_{1}=A, ~\phi_1 =  I, ~\phi_2 =\theta_{1} I, ~N_{1}=0 $, from Equation ($\ref{eq5}$) we obtain 
		\begin{equation*}
			(A+\phi+\theta_{1} I)z^{(k+1)}=\phi z^{(k)}+\lvert(A-\theta_{1} I)z^{(k)}+q\psi(z^{(k)})\rvert-q+\psi(z^{(k)}),	
		\end{equation*}
		this is known as the ``modified  new modified modulus iterative (MNMMOD) method$"$.
		\item when $M_{1}=\frac{1}{\theta_{1}}(D_{1}-\theta_{2} L_{1}), ~ N_{1}=\frac{1}{\theta_{1}}[(1-\theta_{1})D_{1}+(\theta_{1}-\theta_{2})L_{1}+\theta_{1} U_{1}]$, 	where $ D_{1} = diag(A)$ and  $L_{1}$, $U_{1}$ are the strictly lower, upper triangular 
		matrices of $A$, respectively.  From Equation ($\ref{eq5}$), we obtain 
		\begin{equation*}
			\begin{split}
				((D_{\phi_{1}}-\theta_{2} L_{\phi_{1}})+\theta_{1} \phi_{\phi_{1}}+\theta_{1} \phi_{2})z^{(k+1)}&=((1-\theta_{1})D_{\phi_{1}}+(\theta_{1}-\theta_{2})L_{\phi_{1}}\\&+\theta_{1} U_{\phi_{1}}+\theta_{1} \phi_{\phi_{1}})z^{(k)}+\theta_{1}(\lvert(A_{\phi_{1}}\\&-\phi_{2})z^{(k)}+\phi_{1}q+\phi_{2}\psi(z^{(k)})\rvert\\&-\phi_{1}q+\phi_{2}\psi(z^{(k)})),
			\end{split}
		\end{equation*}
	\end{enumerate}
	\noindent this is known as the ``modified  new modulus-based accelerated over relaxation iterative (MNMAOR) method$"$. \\
	The MNMAOR method clearly converts into the  modified new modulus-based successive overrelaxation (MNMSOR) method when $(\theta_1, \theta_2)$ takes the values $(\theta_1, \theta_1)$, modified new modulus-based Gauss-Seidel (MNMGS) method when $(\theta_1, \theta_2)$ takes the values $(1, 1)$ and modified new modulus-based Jacobi (MNMJ) method when $(\theta_1, \theta_2)$ takes the values $(1, 0)$.
	%%%%%%%%%%%%%%%%%%%%%%%%%%%%%%%%%%%%%%%%%%%%%%%%%%%%%%%%%%%%%%%
	\section{Convergence analysis}
	In this section, we discuss the convergence analysis for the Method \ref{mthd1}. First, we discuss the convergence condition when the system matrix $A$ is a $P$-matrix.
	%%%%%%%%%%%%%%%%%%%%%%%%%%%%%%%%%%%%%%%%%%%%%%%%%%%%%%%%%%%%%%%
	\begin{thm}\label{thm1}
		Let $A=(M_{1}+\phi)-(N_{1}+\phi)$ be a splitting of the $P$-matrix $A\in \mathbb{R}^{n\times n}$ and   $M_{\phi_{1}}+\phi_{\phi_{1}}+\phi_{2}$ be nonsingular matrix, where  $\phi$  is a diagonal matrix and  $\phi_{1}$, $\phi_{2}$   are positive diagonal matrices. Let $\rho(L) < 1$, where $L=\lvert(M_{\phi_{1}}+\phi_{\phi_{1}}+\phi_{2})^{-1}\rvert(\lvert N_{\phi_{1}}+\phi_{\phi_{1}}\rvert+\lvert A_{\phi_{1}}-\phi_{2}\rvert+2 \phi_{2}G)$ and $\lvert \psi(x)-\psi(y)\rvert \leq G \lvert x-y \rvert$ $\forall ~x,y \in \mathbb{R}^{n}$, where $G$ is a nonnegative matrix of order $n$.
		Then  the iterative sequence $\{z^{(k)}\}_ {k=0}^{+\infty} \subset \mathbb{R}^n$ generated by Method $\ref{mthd1}$ converges to a unique solution $z^*\in \mathbb{R}^n$  for any nonnegative initial  vector $z^{(0)}\in \mathbb{R}^n$.
	\end{thm}
	\begin{proof}
		Let $z^{*}$ be a solution of the ICP$(q, A,\psi)$.  From Equation($\ref{eq4}$) we obtain
		\begin{equation}\label{eq6}
			\begin{split}
				(M_{\phi_{1}}+\phi_{\phi_{1}}+\phi_{2})z^*&=(N_{\phi_{1}}+\phi_{\phi_{1}})z^*+\lvert(A_{\phi_{1}}-\phi_{2})z^*+\phi_{1}q+\phi_{2}\psi(z^*)\rvert\\&-\phi_{1}q+\phi_{2}\psi(z^*).
			\end{split}
		\end{equation}
		From Equation($\ref{eq4}$) and Equation ($\ref{eq6}$), we obtain
		\begin{equation*}
			\begin{split}
				(M_{\phi_{1}}+\phi_{\phi_{1}}+\phi_{2})(z^{(k+1)}-z^{*})&=(N_{\phi_{1}}+\phi_{\phi_{1}})(z^{(k)}-z^{*})+(\lvert(A_{\phi_{1}}-\phi_{2})z^{(k)}+\phi_{1}q\\&+\phi_2\psi(z^{(k)})\rvert-\lvert(A_{\phi_{1}}-\phi_{2})z^*+\phi_{1}q+\phi_{2}\psi(z^*)\rvert)\\&+(\phi_2\psi(z^{(k)})-\phi_2\psi(z^*)).
			\end{split}
		\end{equation*}
		\noindent By applying   both side modulus, it follows that 
		\begin{equation*}
			\begin{split}
				&\leq \lvert(M_{\phi_{1}}+\phi_{\phi_{1}}+\phi_{2})^{-1}\rvert\lvert[(N_{\phi_{1}}+\phi_{\phi_{1}})(z^{(k)}-z^{*})\rvert+(\lvert(A_{\phi_{1}}-\phi_{2})z^{(k)}\\&+\phi_{1}q+\phi_2\psi(z^{(k)})\rvert-\lvert(A_{\phi_{1}}-\phi_{2})z^*+\phi_{1}q+\phi_{2}\psi(z^*)\rvert)+\phi_2\psi(z^{(k)})\\&-\phi_2\psi(z^*)]\rvert\\
				&\leq \lvert(M_{\phi_{1}}+\phi_{\phi_{1}}+\phi_{2})^{-1}\rvert[\lvert N_{\phi_{1}}+\phi_{\phi_{1}}\rvert+\lvert A_{\phi_{1}}-\phi_{2}\rvert+2\phi_2G]\lvert z^{(k)}-z^*\rvert
				\\&\leq L \lvert z^{(k)}-z^*\rvert
			\end{split}
		\end{equation*}	
		where $L=\lvert(M_{\phi_{1}}+\phi_{\phi_{1}}+\phi_{2})^{-1}\rvert(\lvert N_{\phi_{1}}+\phi_{\phi_{1}}\rvert+\lvert A_{\phi_{1}}-\phi_{2}\rvert+2\phi_2G)$.
		Since $ \rho(L ) < 1$, then
		$$\lvert z^{(k+1)}-z^*\rvert <  \lvert z^{(k)}-z^*\rvert.$$
		Hence, for any nonnegative initial vector $z^{(0)}\in \mathbb{R}^n$ the iterative sequence $\{z^{(k)}\}_{k=0}^{+\infty} \subset \mathbb{R}^n$ generated by Method $\ref{mthd1}$ converges to a unique solution $z^{*} \in \mathbb{R}^n$.
	\end{proof}
	
	%%%%%%%%%%%%%%%%%%%%%%%%%%%%%%%%%%%%%%%%%%%%%%%%%%%%%%%%%%%%%%%

	%%%%%%%%%%%%%%%%%%%%%%%%%%%%%%%%%%%%%%%%%%%%%%%%%%%%%%%%%%%%%%%
	
	%%%%%%%%%%%%%%%%%%%%%%%%%%%%%%%%%%%%%%%%%%%%%%%%%%%%%%%%%%%%%%%
	
	%%%%%%%%%%%%%%%%%%%%%%%%%%%%%%%%%%%%%%%%%%%%%%%%%%%%%%%%%%%%%%%

	\noindent  In the following results, we discuss some convergence conditions for Method \ref{mthd1} when system matrix $A$ is an $H_+$-matrix.
	\begin{thm}\label{thm2}
		Let $A=(M_{1}+\phi)-(N_{1}+\phi)$ be an $H$-splitting of the $H_{+}$-matrix $A\in \mathbb{R}^{n\times n}$. Let  $\phi_{2}\geq D_{\phi_{1}}$, where $D_{\phi_{1}}=\phi_{1}D_{1}$ and $\phi_{1}$,   $\phi_{2}$  are positive diagonal matrices. Let $\phi=(\phi_{ij})$, $\phi_{1}=(w_{ij})$, $\phi_2=(\xi_{ij})$, $M_{1}=(m_{ij})$, $N_{1}=(n_{ij}) \in R^{n \times n}$ and    $G=(g_{ij}) \in R^{n \times n}$ with $g_{ii}\leq \frac{w_{ii}\lvert n_{ii}+\phi_{ii}\rvert - (w_{ii}n_{ii}+\lvert w_{ii}\phi_{ii}\rvert )}{2\xi_{ii}}$
		and $g_{ij}\leq \frac{\lvert w_{ij}m_{ij}\rvert+\lvert w_{ij}n_{ij}\rvert-\lvert w_{ij}a_{ij}\rvert}{2\xi_{ii}}$ for $i,j=1,2,\ldots,n$. Then the iterative sequence $\{z^{(k)}\}_{k=0}^{+\infty} \subset \mathbb{R}^n$ generated by Method $\ref{mthd1}$ converges to a unique solution $z^{*} \in \mathbb{R}^n$ for any nonnegative initial vector $z^{(0)}\in \mathbb{R}^n$.
	\end{thm}	
	\begin{proof}
		Since $\langle M_{1}+\phi\rangle-\lvert N_{1}+\phi\rvert$  is an $M$-matrix and 
		$$ \langle M_{1}+\phi\rangle-\lvert N_{1}+\phi\rvert \leq  \langle M_{1}+\phi\rangle, $$	
		\noindent matrix  $\langle M_{1}+\phi\rangle$ is an $M$-matrix  and $\phi_{2}\geq \phi_{1}D_{1}$, we can obtain that $M_{\phi_{1}}+\phi_{\phi_{1}}+\phi_{2}$ is an $H_{+}$-matrix and $$\lvert(M_{\phi_{1}}+\phi_{\phi_{1}}+\phi_{2})^{-1}\rvert\leq (\langle M_{\phi_{1}}\rangle+\phi_{\phi_{1}}+\phi_{2})^{-1}.$$
		Now we show that matrix 	$\langle M_{\phi_{1}}\rangle+\lvert\phi_{\phi_{1}}\rvert+\phi_{2}-\lvert N_{\phi_{1}}+\phi_{\phi_{1}}\rvert-\lvert A_{\phi_{1}}-\phi_{2}\rvert-2\phi_2G$ is a nonsingular $M$-matrix. First, we calculate 
		
		\noindent	for $i=j$,\\
		$\lvert w_{ii}m_{ii} \rvert + \lvert w_{ii}\phi_{ii} \rvert+\xi_{ii}-\lvert w_{ii}n_{ii}+w_{ii}\phi_{ii}\rvert-\lvert w_{ii}a_{ii}-\xi_{ii} \rvert -2\xi_{ii}g_{ii}$\\
		$\geq 2\lvert w_{ii}m_{ii}+  w_{ii}\phi_{ii}\rvert -2\lvert w_{ii}n_{ii}+w_{ii}\phi_{ii}\rvert $\\
		$\geq 2( \lvert w_{ii}m_{ii}+  w_{ii}\phi_{ii}\rvert-\lvert w_{ii}n_{ii}+w_{ii}\phi_{ii}\rvert)$\\
		
		\noindent and	for $i\neq j$,\\
		$-\lvert w_{ij}m_{ij}\rvert-\lvert w_{ij}n_{ij}\rvert-\lvert w_{ij}a_{ij}\rvert-2w_{ij}g_{ij}$\\
		$\geq -\lvert w_{ij}m_{ij}\rvert-\lvert w_{ij}n_{ij}\rvert-\lvert w_{ij}a_{ij}\rvert-(\lvert w_{ij}m_{ij}\rvert+\lvert w_{ij}n_{ij}\rvert-\lvert w_{ij}a_{ij}\rvert)$\\
		$\geq -(\lvert w_{ij}m_{ij}\rvert+\lvert w_{ij}n_{ij}\rvert).$\\ 
		It follows that \\
		$\langle M_{\phi_{1}}\rangle+\lvert\phi_{\phi_{1}}\rvert+\phi_{2}-\lvert N_{\phi_{1}}+\phi_{\phi_{1}}\rvert-\lvert A_{\phi_{1}}-\phi_{2}\rvert-2\phi_2G \geq 2(\langle M_{1}+\phi\rangle-\lvert N_{1}+\phi\rvert).$\\ 
		We have $\langle M_{1}+\phi\rangle-\lvert N_{1}+\phi\rvert$ is an $M$-matrix. Now we can obtain	$\langle M_{\phi_{1}}\rangle+\lvert\phi_{\phi_{1}}\rvert+\phi_{2}-\lvert N_{\phi_{1}}+\phi_{\phi_{1}}\rvert-\lvert A_{\phi_{1}}-\phi_{2}\rvert-2\phi_2G$ is a nonsingular $M$-matrix.\\
		From Theorem $\ref{thm1}$, we obtain
		\begin{equation*}
			\begin{split}
				\lvert z^{(k+1)}-z^{*}\rvert&=\lvert(M_{\phi_{1}}+\phi_{\phi_{1}}+\phi_{2})^{-1}\rvert(\lvert N_{\phi_{1}}+\phi_{\phi_{1}}\rvert+\lvert A_{\phi_{1}}-\phi_{2}\rvert+2\phi_2G)\lvert z^{(k)}-z^{*}\rvert\\&
				\leq
				(\langle M_{\phi_{1}}\rangle+\lvert\phi_{\phi_{1}}\rvert+\phi_{2})^{-1}(\lvert N_{\phi_{1}}+\phi_{\phi_{1}}\rvert+\lvert A_{\phi_{1}}-\phi_{2}\rvert+2\phi_2G)\lvert z^{(k)}-z^{*}\rvert
			\end{split}
		\end{equation*}
		Let   $\bar{L}=\bar{H}^{-1}\bar{F}$, where
		$\bar{H}=\langle M_{\phi_{1}}\rangle+\lvert\phi_{\phi_{1}}\rvert+\phi_{2}$ and $\bar{F}=\lvert N_{\phi_{1}}+\phi_{\phi_{1}}\rvert+\lvert A_{\phi_{1}}-\phi_{2}\rvert+2\phi_2G$.\\
		We now have \\ 
		$\bar{H}-\bar{F}=\langle M_{\phi_{1}}\rangle+\lvert\phi_{\phi_{1}}\rvert+\phi_{2}-\lvert N_{\phi_{1}}+\phi_{\phi_{1}}\rvert-\lvert A_{\phi_{1}}-\phi_{2}\rvert-2\phi_2G$ and $\bar{H}$ are $M$-matrices and $\bar{F}\geq 0$.\\
		It follows that, the splitting $\bar{H}-\bar{F}$ is an $M$-splitting.\\
		Here, we are able to determine $\rho(\bar{L})< 1$. According to Theorem {\ref{thm1}}, the iterative sequence $\{z^{(k)}\}_{k=0}^{+\infty} \subset \mathbb{R}^n$ generated by Method $\ref{mthd1}$ converges to a unique solution $z^{*} \in \mathbb{R}^n$ for any nonnegative initial vector $z^{(0)}\in \mathbb{R}^n$.
	\end{proof}
	%%%%%%%%%%%%%%%%%%%%%%%%%%%%%%%%%%%%%%%%%%%%%%%%%%%%%%%%%%%%%%%
	\begin{thm}\label{thm3}
		Let $A=(M_{1}+\phi)-(N_1+\phi)= D_{1}-L_1-U_1 =D_1-B$ be a splitting of the $H_{+}$-matrix $A\in \mathbb{R}^{n\times n}$, where $B=L_1+U_{1}$  and $\rho(D_{\phi_{1}}^{-1}\lvert B_{\phi_{1}}\rvert+\phi_2G) < 1$. Let  $\phi_{2}\geq D_{\phi_{1}}$, where $G$ is a nonnegative matrix, $\phi$ is a diagonal matrix and $\phi_{1}$ ,   $\phi_{2}$  are positive diagonal matrices. If the parameters $\theta_{1}$, $\theta_{2}$  satisfy 
		\begin{equation}\label{eq7}
			0 ~\leq~ \max\{\theta_{1}, \theta_{2}\}\rho(D_{\phi_{1}}^{-1}\lvert B_{\phi_{1}}\rvert+\phi_2G) < \min\{1, \theta_{1}\},
		\end{equation} 
		then the iterative sequence $\{z^{(k)}\}_ {k=0}^{+\infty} \subset \mathbb{R}^n$ generated by MNMAOR method converges to a unique solution $z^*\in \mathbb{R}^n$ for  any nonnegative initial vector $z^{(0)}\in \mathbb{R}^n$.
	\end{thm}
	\begin{proof} Following the  proof of the Theorem $\ref{thm2}$, we have \\
		$$\bar{H}-\bar{F}=\langle M_{\phi_{1}}\rangle+\lvert\phi_{\phi_{1}}\rvert+\phi_{2}-\lvert N_{\phi_{1}}+\phi_{\phi_{1}}\rvert-\lvert A_{\phi_{1}}-\phi_{2}\rvert-2\phi_2G.$$
		Let $M_{1}=\frac{1}{\theta_{1}}(D_{1}-\theta_{2} L_{1})$  and $N_{1}=\frac{1}{\theta_{1}}((1-\theta_{1})D_{1}+(\theta_{1}-\theta_{2})L_{1}+\theta_{1} U_{1}).$ Then
		\begin{equation*}
			\begin{split}
				\bar{H}&=\langle M_{\phi_{1}}\rangle+\lvert\phi_{\phi_{1}}\rvert+\phi_{2}\\&
				=\langle \frac{1}{\theta_{1}}(D_{\phi_1}-\theta_{2} L_{\phi_1})\rangle+\lvert\phi_{\phi_{1}}\rvert+\phi_{2}\\&
				= \frac{1}{\theta_{1}}[(D_{\phi_1}-\theta_{2} \lvert L_{\phi_1}\rvert)+\theta_{1}\lvert\phi_{\phi_{1}}\rvert+\theta_{1}\phi_{2}]
			\end{split}
		\end{equation*}
		and
		\begin{equation*}
			\begin{split}
				\bar{F}=\frac{1}{\theta_{1}}[\lvert(1-\theta_{1})D_{\phi_1}+(\theta_{1}-\theta_{2})L_{\phi_1}+\theta_{1} U_{\phi_1}+\theta_{1}\phi_{\phi_{1}}\rvert+\theta_{1}\lvert A_{\phi_{1}}-\phi_{2}\rvert+2\theta_{1}\phi_2G].
			\end{split}
		\end{equation*}
		Now we compute
		\begin{equation*}
			\begin{split}
				\theta_1(\bar{H}-\bar{F})&=[(D_{\phi_1}-\theta_{2} \lvert L_{\phi_1}\rvert)+\theta_{1}\lvert\phi_{\phi_{1}}\rvert+\theta_{1}\phi_{2}]-[\lvert(1-\theta_{1})D_{\phi_1}\\&+(\theta_{1}-\theta_{2})L_{\phi_1}+\theta_{1} U_{\phi_1}+\theta_{1}\phi_{\phi_{1}}\rvert+\theta_{1}\lvert A_{\phi_{1}}-\phi_{2}\rvert+2\theta_{1}\phi_2G]\\&
				\geq  (1+\theta_{1}-\lvert1-\theta_{1}\rvert)D_{\phi_{1}} -\lvert\theta_{1}B_{\phi_{1}}-\theta_{2} L_{1}\rvert\\&-\theta_{1}\lvert B_{\phi_{1}}\rvert-\theta_{2} \lvert L_{\phi_1}\rvert-2\theta_{1}\phi_2G.
			\end{split}
		\end{equation*}
		Since  
		\begin{equation*}
			(1 +\theta_{1} -\lvert1-\theta_{1} \rvert)= 2\min\{1, \theta_{1}\}
		\end{equation*} 
		and 
		\begin{equation*}
			\begin{split}
				&\lvert\theta_{1}B_{\phi_{1}}-\theta_{2} L_{\phi_1}\rvert+\theta_{1}\lvert B_{\phi_{1}}\rvert+\theta_{2} \lvert L_{\phi_1}\rvert
				\\&\leq \theta_{1}\lvert B_{\phi_{1}}\rvert+\theta_{2} \lvert L_{\phi_1}\rvert+\theta_{1}\lvert B_{\phi_{1}}\rvert+\theta_{2} \lvert L_{\phi_1}\rvert
				\\&\leq 2\max\{\theta_{1},\theta_{2}\}\lvert B_{\phi_{1}}\rvert.
			\end{split}
		\end{equation*}
		At this point, \begin{equation*}
			\begin{split}
				\theta_1(\bar{H}-\bar{F})&\geq 2\min\{1, \theta_{1}\}D_{\phi_1}-2\max\{\theta_{1},\theta_{2}\}\lvert B_{\phi_{1}}\rvert-2\theta_{1}\phi_2G\\&
				\geq 2D_{\phi_1}(\min\{1, \theta_{1}\}I-\max\{\theta_{1},\theta_{2}\}D^{-1}_{\phi_1}(\lvert B_{\phi_{1}}\rvert+\phi_2G).
			\end{split}
		\end{equation*}
		Since $\rho(D^{-1}_{\phi_1}(\lvert B_{\phi_{1}}\rvert+\phi_2G)< 1$. It is easy to verify that $\theta_1(\bar{H}-\bar{F})$ an $M$-matrix. Further, we obtain that matrix $\bar{H}-\bar{F}$ is an $M$-matrix.
	\end{proof}
	\begin{thm}
		Let $A=(M_{1}+\phi)-(N_{1}+\phi)$ be an $H$-splitting with the $H_+$-matrix $A \in \mathbb{R}^{n\times n}$, matrix $G$ a nonnegative matrix  and  $D = diag(A)$. Then the Method \ref{mthd1} is convergent,  provided that for some positive diagonal matrix $U$ such that $(\langle M_{1}+\phi\rangle -(N_{1}+\phi))U$  is a strictly diagonally dominant matrix,  one of the following conditions is valid:\\
		(1) $\phi_{2} \geq D_{\phi_{1}}$,\\
		$$(\langle M_{\phi_{1}}+\phi\rangle -\lvert N_{\phi_{1}}+\phi\rvert -\phi_{2}G)Ue > 0;$$
		(2) $0 < \phi_{2}< D_{\phi_{1}}$, $$(\phi_{2}-D_{\phi_{1}}+\langle M_{\phi_{1}}+\phi\rangle -\lvert N_{\phi_{1}}+\phi\rvert -\phi_{2}G)Ue > 0.$$ 
	\end{thm}
	\begin{proof}
		Let $\bar{A}=\langle M_{1}+\phi\rangle -\lvert N_{1}+\phi\rvert $, $\phi_{1}\bar{A}=\bar{A}_{\phi_{1}}$ and $A=(M_{1}+\phi)-(N_{1}+\phi)$ be an $H$-splitting. Therefore $\bar{A}$ is a nonsingular  $M$-matrix.\\
		$$0 < \bar{A} Ue \leq \langle M_{1}+\phi\rangle Ue \leq (\langle M_{1}+\phi\rangle +\phi_{2}) Ue  $$
		and $$0 < \bar{A}Ue < \langle A\rangle Ue $$
		Now we calculate 
		\begin{equation}\label{eq8}
			\begin{split}
				&(\langle M_{\phi_{1}}\rangle+\lvert\phi_{\phi_{1}}\rvert+\phi_{2}-\lvert N_{\phi_{1}}+\phi_{\phi_{1}}\rvert-\lvert A_{\phi_{1}}-\phi_{2}\rvert-2\phi_2G)Ue\\&
				\geq 
				(\langle M_{\phi_{1}}+\phi_{\phi_{1}}\rangle-\lvert N_{\phi_{1}}+\phi_{\phi_{1}}\rvert+\phi_{2}-\lvert A_{\phi_{1}}-\phi_{2}\rvert-2\phi_2G)Ue.
			\end{split}
		\end{equation}
		Here, we consider two cases\\
		Case (1): when $\phi_{2}\geq D_{\phi_{1}}$,
		\begin{equation*}
			\begin{split}
				(\bar{A}_{\phi_{1}}+\phi_{2}-\lvert A_{\phi_{1}}-\phi_{2}\rvert-2\phi_2Q)Ue
				&= (\bar{A}_{\phi_{1}}+\phi_{2}-( \phi_{2}-\langle A_{\phi_{1}}\rangle)-2\phi_2G)Ue\\&
				\geq 2(\bar{A}_{\phi_{1}}-\phi_2G)Ue\\&
				\geq 0.
			\end{split}
		\end{equation*}
		Case (1): when $0 < \phi_{2}\leq D_{\phi_{1}}$,
		\begin{equation*}
			\begin{split}
				(\bar{A}_{\phi_{1}}+\phi_{2}-\lvert A_{\phi_{1}}-\phi_{2}\rvert-2\phi_2G)Ue
				&= (\bar{A}_{\phi_{1}}+\phi_{2}-(\lvert A_{\phi_{1}}\rvert-\phi_{2})-2\phi_2G)Ue\\&
				= (\bar{A}_{\phi_{1}}+\phi_{2}-(2D_{\phi_{1}}-\langle A_{\phi_{1}}\rangle-\phi_{2})-2\phi_2G)Ue\\&
				\geq 2(\bar{A}_{\phi_{1}}+\phi_{2}-D_{\phi_{1}}-\phi_2G)Ue\\&
				\geq 0.
			\end{split}
		\end{equation*}
		Now from case (1), case (2) and Equation$(\ref{eq8})$, we obtain that 
		$$(\langle M_{\phi_{1}}\rangle+\lvert\phi_{\phi_{1}}\rvert+\phi_{2}-\lvert N_{\phi_{1}}+\phi_{\phi_{1}}\rvert-\lvert A_{\phi_{1}}-\phi_{2}\rvert-2\phi_2G)Ue > 0.$$
		Now from Lemma $\ref{2.4}$, we have 
		\begin{equation*}
			\begin{split}
				\rho(\bar{L})&=\rho(U^{-1}\bar{L}U)\\&
				\leq \|U^{-1}\bar{L}U\|_{\infty}\\&
				=\| U^{-1}(\langle M_{\phi_{1}}\rangle+\lvert\phi_{\phi_{1}}\rvert+\phi_{2})^{-1}(\lvert N_{\phi_{1}}+\phi_{\phi_{1}}\rvert+\lvert A_{\phi_{1}}-\phi_{2}\rvert+2\phi_2G)U\|_{\infty}\\&
				\leq \max_{1\leq i \leq n}\frac{((\lvert N_{\phi_{1}}+\phi_{\phi_{1}}\rvert+\lvert A_{\phi_{1}}-\phi_{2}\rvert+2\phi_2G)Ue)_{i}}{((\langle M_{\phi_{1}}\rangle+\lvert\phi_{\phi_{1}}\rvert+\phi_{2})Ue)_{i}}\\&
				< 1.
			\end{split}
		\end{equation*}
		Hence, the Method \ref{mthd1} is convergent.
	\end{proof} 
	\section{Conclusion} 
	We introduce a class of modified new modulus-based matrix splitting methods based on a new matrix splitting approach for solving the implicit complementarity problem ICP$(q, A, \psi)$ with parameter matrices  $\phi_1$ and $\phi_2$.  These methods are able to process the large and sparse structure of the matrix $A$.  For the proposed methods, we establish some convergence conditions when the system matrix is a $P$-matrix as well as some sufficient convergence conditions when the system matrix is an $H_{+}$-matrix. 
	%\nocite{oreg,schn,pond,smith,marg,hunn,advi,koha,mouse}
	%%%%%%%%%%%%%%%%%%%%%%%%%%%%%%%%%%%%%%%%%%%%%%
	%%                                          %%
	%% Backmatter begins here                   %%
	%%                                          %%
	%%%%%%%%%%%%%%%%%%%%%%%%%%%%%%%%%%%%%%%%%%%%%%
	
	\subsection*{Acknowledgments}
	The author, Bharat Kumar is thankful to the University Grants Commission (UGC), Government of India, under the SRF fellowship, Ref. No.: 1068/(CSIR-UGC NET DEC. 2017).
	%\section*{Authors' information}%% if any
	%Text for this section\ldots
	%%  After submission of the .TEX file,                     %%
	%%  you will be prompted to submit your .BBL file.         %%
	% if your bibliography is in bibtex format, use those commands:
	\bibliographystyle{plain} % Style BST file (bmc-mathphys, vancouver, spbasic).
	%	\bibliography{bibfile}      % Bibliography file (usually '*.bib' )
	% for author-year bibliography (bmc-mathphys or spbasic)
	% a) write to bib file (bmc-mathphys only)
	% @settings{label, options="nameyear"}
	% or include bibliography directly:
	% \begin{thebibliography}
		% \bibitem{b1}
		% \end{thebibliography}
	%\bibliographystyle{sn-basic}
	\bibliography{nam}
	
\end{document}